\documentclass{amsart}

\usepackage{amssymb,amsthm,bm,mathrsfs}
\usepackage{hyperref}

\newcommand{\arxiv}[1]{\href{http://arxiv.org/abs/#1}{arXiv:#1}}

\newcommand{\myoverline}[1]{\overline{#1}\vphantom{#1}}
\newcommand{\kvec}{\mathfrak{k}}
\newcommand{\rvec}{\mathfrak{r}}
\newcommand{\svec}{\mathfrak{s}}

\newtheorem{thm}{Theorem}[section]
\newtheorem{lem}[thm]{Lemma}
\newtheorem{prop}[thm]{Proposition}
\theoremstyle{definition}
\newtheorem*{notn}{Notation \& Assumption}
\newtheorem{ex}[thm]{Example}
\newtheorem{rem}[thm]{Remark}

\begin{document}

\title[Asymptotic joint spectra]{Asymptotic joint spectra of Cartesian powers of strongly regular graphs and bivariate Charlier--Hermite polynomials}

\keywords{Quantum probability, central limit theorem, spectral distribution, strongly regular graph, orthogonal polynomial, hypergeometric series}
\subjclass[2010]{Primary 46L53; Secondary 60F05, 05E30, 33C45.}

\author{John Vincent S. Morales}
\address{Mathematics and Statistics Department, College of Science, De La Salle University, Manila, Philippines}
\email{john.vincent.morales@dlsu.edu.ph}
\author{Nobuaki Obata}
\address{\href{http://www.math.is.tohoku.ac.jp/}{Research Center for Pure and Applied Mathematics}, Graduate School of Information Sciences, Tohoku University, Sendai, Japan}
\email{obata@math.is.tohoku.ac.jp}
\urladdr{http://www.math.is.tohoku.ac.jp/~obata/}
\author{Hajime Tanaka}
\address{\href{http://www.math.is.tohoku.ac.jp/}{Research Center for Pure and Applied Mathematics}, Graduate School of Information Sciences, Tohoku University, Sendai, Japan}
\email{htanaka@tohoku.ac.jp}
\urladdr{http://www.math.is.tohoku.ac.jp/~htanaka/}

\begin{abstract}
Generalizing previous work of Hora (1998) on the asymptotic spectral analysis for the Hamming graph $H(n,q)$ which is the $n^{\mathrm{th}}$ Cartesian power $K_q^{\square n}$ of the complete graph $K_q$ on $q$ vertices, we describe the possible limits of the joint spectral distribution of the pair $(G^{\square n},\myoverline{G}^{\square n})$ of the $n^{\mathrm{th}}$ Cartesian powers of a strongly regular graph $G$ and its complement $\overline{G}$, where we let $n\rightarrow\infty$, and $G$ may vary with $n$.
This result is an analogue of the bivariate central limit theorem, and we obtain in this way the bivariate Poisson distributions and the standard bivariate Gaussian distribution, together with the product measures of univariate Poisson and Gaussian distributions.
We also report a family of bivariate hypergeometric orthogonal polynomials with respect to the last distributions, which we call the \emph{bivariate Charlier--Hermite polynomials}, and prove basic formulas for them.
This family of orthogonal polynomials seems previously unnoticed, possibly because of its peculiarity.
\end{abstract}

\maketitle

\section{Introduction}

Let $G=(V,E)$ be a finite simple graph with vertex set $V$ and edge set $E$.
(Formal definitions about graphs will be given in Section \ref{sec: main result}.)
The \emph{adjacency matrix} $A$ of $G$ is the $0$-$1$ matrix indexed by $V$, where $A_{x,y}=1$ if and only if $x$ and $y$ are adjacent.
By an \emph{eigenvalue} of $G$ we mean an eigenvalue of $A$.
Likewise, we speak of the \emph{spectrum} of $G$.

Spectra of graphs have been receiving attention from the point of view of quantum probability theory.
Recall that an \emph{algebraic probability space} is a pair $(\mathcal{A},\varphi)$, where $\mathcal{A}$ is a $*$-algebra over $\mathbb{C}$ and $\varphi:\mathcal{A}\rightarrow\mathbb{C}$ is a \emph{state}, i.e., a linear map such that $\varphi(1)=1$ and that $\varphi(a^*a)\geqslant 0$ for every $a\in\mathcal{A}$.
The elements of $\mathcal{A}$ are referred to as (\emph{algebraic}) \emph{random variables}.
We call $a\in\mathcal{A}$ \emph{real} if $a^*=a$.
For a real random variable $a\in\mathcal{A}$, we are interested in finding, and discussing the uniqueness of, a probability measure $\nu$ on $\mathbb{R}$ such that
\begin{equation}\label{distribution}
	\varphi(a^j)=\int_{-\infty}^{+\infty}\!\! x^j\,\nu(dx) \quad (j=0,1,2,\dots).
\end{equation}

Associated with the graph $G$ above is the \emph{adjacency algebra} $\mathbb{C}[A]$, i.e., the commutative subalgebra of the full matrix algebra generated by $A$.
For the $*$-algebra $\mathbb{C}[A]$, it is natural to consider the \emph{tracial state} $\varphi_{\mathrm{tr}}$ defined by\footnote{Another important example is the \emph{vacuum state} $\varphi_x(X)=X_{x,x}$ $(X\in\mathbb{C}[A])$ at a fixed origin $x\in V$. We note that the matrix $*$-algebras we will discuss in this paper all have the property that every element has constant diagonal entries, so that the two states $\varphi_{\mathrm{tr}}$ and $\varphi_x$ turn out to be equal on them.}
\begin{equation*}
	\varphi_{\operatorname{tr}}(X)=\frac{1}{|V|}\operatorname{tr}(X) \quad (X\in\mathbb{C}[A]).
\end{equation*}
Suppose for simplicity that $G$ is $k$-regular, so that $A$ has mean $\varphi_{\operatorname{tr}}(A)=0$ and variance $\varphi_{\operatorname{tr}}(A^2)=k$.
Let $k=\theta_0>\theta_1>\dots>\theta_d$ be the distinct eigenvalues of $G$, and let $m_j$ be the multiplicity of $\theta_j$.
Then the unique probability measure $\nu=\nu_G$ in \eqref{distribution} for $a=A/\sqrt{k}$ is the (normalized) \emph{spectral distribution} of $G$ given by
\begin{equation*}
	\nu_G\!\left(\frac{\theta_j}{\sqrt{k}}\right)=\frac{m_j}{|V|} \quad (j=0,1,\dots,d).
\end{equation*}

Our focus is on the limit of $\nu_G$ when $G$ ``grows'', as an analogue of the classical central limit theorem.
Hora \cite{Hora1998IDAQPRT} described various limit distributions for several growing families of Cayley graphs and \emph{distance-regular graphs} (cf.~\cite{DKT2016EJC}).
For example, for the Hamming graphs $H(n,q)$ which are one of the most important families of distance-regular graphs, he obtained a Poisson distribution when $q/n\rightarrow \tau$ $(n\rightarrow\infty)$ where $0<\tau<\infty$, and the standard Gaussian distribution when $q/n\rightarrow 0$ $(n\rightarrow\infty)$.
Hora worked with the spectra directly in \cite{Hora1998IDAQPRT}, but then Hora, Obata, and others revisited, simplified, and generalized these results based on the idea of decomposing the random variable $A$ into the sum of certain three non-commuting components $A^+,A^{\circ},A^-$ in a larger $*$-algebra,\footnote{We may remark that $A^+,A^{\circ}$, and $A^-$ belong to the \emph{Terwilliger algebra} \cite{Terwilliger1992JAC} of $G$. See \cite[Section 16.6]{DKT2016EJC}.} called the \emph{quantum decomposition} of $A$.
Besides the Poisson and Gaussian distributions, many important univariate distributions arise in this way, such as the exponential, geometric, gamma, and the two-sided Rayleigh distributions.
See, e.g., \cite{HHO2003JMP,HOT2001P,HO2007B,HO2008TAMS} for more details.

The purpose of the present paper is to give a concrete \emph{bivariate} example of this sort, as an attempt towards a multivariate extension of the theory.
Consider again a general algebraic probability space $(\mathcal{A},\varphi)$.
We now pick two \emph{commuting} real random variables $a,b\in\mathcal{A}$, and discuss a probability measure $\nu$ on $\mathbb{R}^2$ such that
\begin{equation}\label{joint distribution}
	\varphi(a^jb^h)=\int_{\mathbb{R}^2}x^jy^h\,\nu(dxdy) \quad (j,h=0,1,2,\dots).
\end{equation}
In our context, we take another (say, $\ell$-regular) graph $H=(V,F)$ having the same vertex set $V$ as $G$, and assume that the adjacency matrix $B$ of $H$ commutes with $A$.
This occurs for example when $H=\overline{G}$, the complement of $G$.
(Recall that we are assuming that $G$ is $k$-regular.)
We view $A$ and $B$ as real random variables in the algebraic probability space $(\mathbb{C}[A,B],\varphi_{\operatorname{tr}})$.
Note that the covariance $\varphi_{\operatorname{tr}}(AB)$ for $A$ and $B$ equals $0$ if and only if $G$ and $H$ have no edge in common, i.e., $E\cap F=\emptyset$.
Let $\ell=\eta_0>\eta_1>\dots>\eta_e$ be the distinct eigenvalues of $H$, and let $m_{j,h}$ be the dimension of the common eigenspace of $(A,B)$ with respective eigenvalues $(\theta_j,\eta_h)$.
The probability measure $\nu=\nu_{G,H}$ in \eqref{joint distribution} for $a=A/\sqrt{k}$ and $b=B/\sqrt{\ell}$ is then the (normalized) \emph{joint spectral distribution} of $G$ and $H$ given by
\begin{equation*}
	\nu_{G,H}\!\left(\frac{\theta_j}{\sqrt{k}},\frac{\eta_h}{\sqrt{\ell}}\right)=\frac{m_{j,h}}{|V|} \quad (j=0,1,\dots,d,\ h=0,1,\dots,e).
\end{equation*}
We are again interested in the limit of $\nu_{G,H}$ when $G$ and $H$ both grow, as an analogue of the bivariate central limit theorem.

Our main result (Theorem \ref{main theorem}) is indeed a bivariate version of the result of Hora \cite{Hora1998IDAQPRT} for the Hamming graphs mentioned above.
The Hamming graph $H(n,q)$ is defined as the $n^{\mathrm{th}}$ \emph{Cartesian power} $K_q^{\square n}$ of the complete graph $K_q$ on $q$ vertices.
We will instead consider the pair $(G^{\square n},\myoverline{G}^{\square n})$ of the $n^{\mathrm{th}}$ Cartesian powers of a \emph{strongly regular graph} $G$ and its complement $\overline{G}$, and obtain as limits the bivariate Poisson distributions and the standard bivariate Gaussian distribution, together with the product measures of univariate Poisson and Gaussian distributions.
The method of quantum decomposition is yet to be developed for the multivariate case, and hence we will deal with the spectra of these graphs directly, as was done by Hora in \cite{Hora1998IDAQPRT}, though the discussions here become much more involved.
We note that the complete graphs are the connected regular graphs with precisely two distinct eigenvalues, whereas the connected strongly regular graphs are those with precisely three distinct eigenvalues.
This comparison can be made clearer when viewed in the framework of \emph{association schemes}, and our choice of considering the pair $(G^{\square n},\myoverline{G}^{\square n})$ above was in fact guided naturally by the work of Mizukawa and Tanaka \cite{MT2004PAMS} on a construction of multivariate Krawtchouk polynomials from arbitrary association schemes.
See Section \ref{sec: Charlier-Hermite polynomials}.
As a by-product, we report in Section \ref{sec: Charlier-Hermite polynomials} a family of bivariate hypergeometric orthogonal polynomials with respect to the last distributions, which we call the \emph{bivariate Charlier--Hermite polynomials}, and prove basic formulas for them.
This family of orthogonal polynomials seems previously unnoticed, possibly because of its peculiarity.

Section \ref{sec: strongly regular graphs} collects necessary facts about strongly regular graphs.
Section \ref{sec: proof} is devoted to the proof of Theorem \ref{main theorem}.
In Section \ref{sec: examples}, we demonstrate Theorem \ref{main theorem} with some specific families of strongly regular graphs.

\section{Basic definitions and the main result}\label{sec: main result}

Let $G=(V,E)$ be a graph with vertex set $V$ and edge set $E$.
All the graphs we consider in this paper are finite and simple.
Thus, $V$ is a finite set and $E$ is a subset of $\binom{V}{2}$, the set of $2$-element subsets of $V$.
The elements of $V$ are \emph{vertices} of $G$ and the elements of $E$ are \emph{edges} of $G$.
Two vertices $x,y\in V$ are called \emph{adjacent} (and written $x\sim y$) if $\{x,y\}\in E$.
The \emph{degree} (or \emph{valency}) $k(x)$ of a vertex $x\in V$ is the number of vertices adjacent to $x$.
The graph $G$ is called $k$-\emph{regular} if $k(x)=k$ for all $x\in V$.
It is called \emph{connected} if for any two vertices $x$ and $y$, there is a sequence of vertices $x=x_0,x_1,\dots,x_t=y$ such that $x_{j-1}\sim x_j$ for $j=1,2,\dots,t$.
Recall that a \emph{complete graph} $K_v$ is a graph on $|V|=v$ vertices such that $E=\binom{V}{2}$.
As in Introduction, the \emph{adjacency matrix} $A$ of $G$ is the matrix indexed by $V$ such that $A_{x,y}=1$ if $x\sim y$ and $A_{x,y}=0$ otherwise.
The \emph{complement} $\overline{G}$ of $G$ is the graph with the same vertex set $V$ as $G$, where two distinct vertices are adjacent if and only if they are non-adjacent in $G$.
Thus, $\overline{G}$ has adjacency matrix $\overline{A}:=J-A-I$, where $I$ and $J$ denote the identity matrix and the all-ones matrix, respectively.

The \emph{Cartesian product} $G_1\,\square\, G_2$ of two graphs $G_j=(V_j,E_j)$ $(j=1,2)$ is the graph with vertex set $V_1\times V_2$, where $(x_1,x_2)\sim (y_1,y_2)$ if and only if either $x_1\sim y_1$ and $x_2=y_2$, or $x_1=y_1$ and $x_2\sim y_2$; cf.~\cite[Section 1.4.6]{BH2012B}.
For a positive integer $n$, the Cartesian power $G\,\square\,G\,\square\cdots\square\,G$ ($n$ times) will be denoted by $G^{\square n}$.
For example, we already mentioned that $H(n,q)=K_q^{\square n}$.
The adjacency matrix $\bm{A}$ of $G^{\square n}$ is given by
\begin{equation}\label{adjacency matrix of power}
	\bm{A}=\sum_{j=1}^n I\otimes\dots \otimes I\otimes \underset{\stackrel{\frown\vphantom{A}}{j}}{A} \otimes I\otimes \dots \otimes I.
\end{equation}

From now on, suppose that $G$ is $k$-regular and has $|V|=v$ vertices.
We note that $k$ is an eigenvalue of $G$ (i.e., of $A$), and that every eigenvalue $\theta$ of $G$ satisfies $|\theta|\leqslant k$; cf.~\cite[Section 1.3.1]{BH2012B}.
We have $AJ=JA=kJ$, and hence $A\overline{A}=\overline{A}A$.
Observe that $G^{\square n}$ is $nk$-regular, and that $\bm{A}\overline{\bm{A}}=\overline{\bm{A}}\bm{A}$, where $\overline{\bm{A}}$ denotes the adjacency matrix of $\myoverline{G}^{\square n}$.
Since $G^{\square n}$ and $\myoverline{G}^{\square n}$ have no edge in common, the covariance $\varphi_{\operatorname{tr}}(\bm{A}\overline{\bm{A}})=0$.

We call $G$ \emph{strongly regular} with parameters $(v,k,\lambda,\mu)$ if $G$ is not complete or edgeless (i.e., $0<k<v-1$), and if every pair of adjacent (resp.~non-adjacent and distinct) vertices has precisely $\lambda$ (resp.~$\mu$) common adjacent vertices; cf.~\cite[Section 9.1]{BH2012B}.
In matrix terms, this means that
\begin{equation}\label{quadratic equation}
	A^2=kI+\lambda A+\mu \overline{A}.
\end{equation}
It is clear that $G$ is a disconnected strongly regular graph precisely when it is the disjoint union $p K_q$ of $p$ complete graphs $K_q$ for some integers $p,q\geqslant 2$.
It is easy to see that if $G$ is strongly regular as above then $\overline{G}$ is again strongly regular with parameters $(v,\overline{k},\overline{\lambda},\overline{\mu})$, where
\begin{equation}\label{parameters of complement}
	\overline{k}=v-k-1, \quad \overline{\lambda}=v-2k+\mu-2, \quad \overline{\mu}=v-2k+\lambda.
\end{equation}
Thus, strongly regular graphs always exist in pairs.
The complement of $pK_q$ is the \emph{complete multipartite graph} $K_{p\times q}$.

Observe that $G$ is complete if and only if the linear span $\langle I,A\rangle$ equals $\langle I,J\rangle$, which is a two-dimensional $*$-algebra.
Likewise, from \eqref{quadratic equation} it follows that $G$ is strongly regular as above if and only if $\langle I,A,\overline{A}\rangle=\langle I,A,J\rangle$ is a three-dimensional (commutative) $*$-algebra.
Suppose now that this is the case.
Then it follows that there are exactly three (maximal) common eigenspaces for $(A,\overline{A})$, one of which corresponds to the eigenvalues $(k,\overline{k})$ and is spanned by the all-ones vector $\bm{1}$ in $\mathbb{C}^v$.
Let $(r,\overline{s})$ and $(s,\overline {r})$ denote the eigenvalues corresponding to the other two, where we have $\overline{s}=-r-1$ and $\overline{r}=-s-1$.
We will assume that $r>s$, or equivalently, $\overline{r}>\overline{s}$.
We have $s,\overline{s}<0$ since $\operatorname{tr}(A)=\operatorname{tr}(\overline{A})=0$, so that\footnote{In fact, it follows that $r,\overline{r}\geqslant 0$ and $s,\overline{s}\leqslant -1$; cf.~Lemma \ref{half case}.}
\begin{equation}\label{range of eigenvalues}
	-1<r\leqslant k, \quad -k\leqslant s<0, \quad -1<\overline{r}\leqslant \overline{k}, \quad -\overline{k}\leqslant \overline{s}<0.
\end{equation}
We call $r$ and $s$ (resp.~$\overline{r}$ and $\overline{s}$) the \emph{restricted}\footnote{More generally, an eigenvalue of a (not necessarily regular) graph is called \emph{restricted} if it has an eigenvector which is not a scalar multiple of $\bm{1}$.} eigenvalues of the strongly regular graph $G$ (resp.~$\overline{G}$).

\begin{notn}
We consider an infinite family of pairs of Cartesian powers of graphs $(G^{\square n},\myoverline{G}^{\square n})$, where $n$ ranges over an infinite set of positive integers, and $G$ is strongly regular and may vary with $n$.
To simplify notation, we think of $G,v,k,\overline{k},r,s$, etc., as functions in $n$.
We will assume that
\begin{equation*}
	\frac{k}{n}\rightarrow \kappa, \quad \frac{\overline{k}}{n}\rightarrow \overline{\kappa}, \quad \frac{r}{n}\rightarrow \rho, \quad \frac{s}{n}\rightarrow \sigma
\end{equation*}
as $n\rightarrow\infty$, where $\kappa,\overline{\kappa},\rho$, and $\sigma$ are finite.
We note that
\begin{equation*}
	\frac{v}{n}\rightarrow \omega:=\kappa+\overline{\kappa}.
\end{equation*}
\end{notn}

The following is our main result which describes the possible limits of the joint spectral distribution $\nu_{G^{\square n},\,\myoverline{G}^{\square n}}$.

\begin{thm}\label{main theorem}
With the above notation and assumption, we have $\rho=0$ or $\sigma=0$, and 
one of the following holds:
\begin{enumerate}
\item $\kappa>0$, $\overline{\kappa}=-\sigma>0$, $\rho=0$, and $\nu_{G^{\square n},\,\myoverline{G}^{\square n}}$ converges weakly to an affine transformation $\nu$ of a bivariate Poisson distribution given by
\begin{equation*}
	\nu\!\left(\frac{\kappa j-\overline{\kappa} h}{\sqrt{\kappa}},\frac{\overline{\kappa} j+\overline{\kappa} h-1}{\sqrt{\overline{\kappa}}}\right)\!=e^{-1/\overline{\kappa}}\!
	\left(\frac{1}{\omega}\right)^{\!\!j}\!\!\left( \vphantom{\bigg(} \frac{\kappa}{\omega\overline{\kappa}}\right)^{\!\!h}\!\frac{1}{j!h!}
\end{equation*}
for $j,h=0,1,2,\dots$.
In this case, $G$ is a complete multipartite graph for all but finitely many values of $n$.
\item $\kappa=\rho>0$, $\overline{\kappa}>0$, $\sigma=0$, and $\nu_{G^{\square n},\,\myoverline{G}^{\square n}}$ converges weakly to an affine transformation $\nu$ of a bivariate Poisson distribution given by
\begin{equation*}
	\nu\!\left(\frac{\kappa j+\kappa h-1}{\sqrt{\kappa}},\frac{\overline{\kappa}j-\kappa h}{\sqrt{\overline{\kappa}}}\right)\!=e^{-1/\kappa}\!
	\left(\frac{1}{\omega}\right)^{\!\!j}\!\!\left(\frac{\overline{\kappa}}{\omega\kappa}\right)^{\!\!h}\!\frac{1}{j!h!}
\end{equation*}
for $j,h=0,1,2,\dots$.
In this case, $G$ is a disjoint union of complete graphs for all but finitely many values of $n$.
\item $\kappa>0$ or $\overline{\kappa}>0$, and $\rho=\sigma=0$, and $\nu_{G^{\square n},\,\myoverline{G}^{\square n}}$ converges weakly to an affine transformation $\nu$ of the product measure of a Poisson distribution and a Gaussian distribution given by
\begin{equation*}
	\int_{\mathbb{R}^2}\!\!\gamma(x,y)\hspace{.01in}\nu(dxdy)=\!\sqrt{\!\frac{\omega}{2\pi}}\,e^{-1/\omega}\sum_{h=0}^{\infty}\!\left(\frac{1}{\omega}\right)^{\!\!h}\!\frac{1}{h!}\! \int_{-\infty}^{+\infty}\!\!\!\! \gamma(\bm{z}_{h,t})\hspace{.01in}e^{-\omega t^2/2}\hspace{.01in}dt
\end{equation*}
for every Borel function $\gamma:\mathbb{R}^2\rightarrow\mathbb{R}$, where
\begin{equation*}
	\bm{z}_{h,t}=\left(\!\sqrt{\kappa}\,h+\!\sqrt{\overline{\kappa}}\,t-\frac{\sqrt{\kappa}}{\omega},\sqrt{\overline{\kappa}}\,h-\!\sqrt{\kappa}\,t-\frac{\sqrt{\overline{\kappa}}}{\omega}\right).
\end{equation*}
\item $\kappa=\overline{\kappa}=\rho=\sigma=0$, and $\nu_{G^{\square n},\,\myoverline{G}^{\square n}}$ converges weakly to the standard bivariate Gaussian distribution.
\end{enumerate}
\end{thm}

\section{Preliminaries on strongly regular graphs}\label{sec: strongly regular graphs}

In this section, we collect necessary facts about strongly regular graphs.
See \cite[Chapter 9]{BH2012B} for more details.
Throughout this section, let $G$ be a (fixed) strongly regular graph with parameters $(v,k,\lambda,\mu)$, and let $\overline{G}$ be the complement of $G$, having parameters $(v,\overline{k},\overline{\lambda},\overline{\mu})$ (cf.~\eqref{parameters of complement}).
Let $A$ (resp.~$\overline{A}$) be the adjacency matrix of $G$ (resp.~$\overline{G}$).
For convenience, we let
\begin{equation*}
	\kvec=(k,\overline{k}), \quad \rvec=(r,\overline{s}), \quad \svec=(s,\overline{r}).
\end{equation*}
Let $\mathbb{U}_{\kvec},\mathbb{U}_{\rvec}$, and $\mathbb{U}_{\svec}$ be the common eigenspaces of $(A,\overline{A})$ associated with $\kvec,\rvec$, and $\svec$, respectively.
Recall that $\mathbb{U}_{\kvec}=\langle\bm{1}\rangle$.
Let
\begin{equation*}
	f=\dim(\mathbb{U}_{\rvec}), \quad g=\dim(\mathbb{U}_{\svec}).
\end{equation*}

There are a number of standard identities involving these scalars.
What we will need in the proof of Theorem \ref{main theorem} are the following:
\begin{gather}
	v=1+k+\overline{k}=1+f+g, \label{sum to v} \\
	0=1+r+\overline{s}=1+s+\overline{r}, \label{row sums of P} \\
	0=k+rf+sg=\overline{k}+\overline{s}f+\overline{r}g, \label{linear traces} \\
	k^2+r^2f+s^2g=kv, \label{quadratic trace 1} \\
	k\overline{k}+r\overline{s}f+s\overline{r}g=0, \label{quadratic trace 2} \\
	\myoverline{k}^2+\myoverline{s}^2f+\myoverline{r}^2g=\overline{k}v, \label{quadratic trace 3} \\
	f=\frac{(v-1)s+k}{s-r},  \quad g=\frac{(v-1)r+k}{r-s}, \label{f and g} \\
	fg=\frac{k\overline{k}v}{(r-s)^2}. \label{fg}
\end{gather}
The proofs of \eqref{sum to v}--\eqref{f and g} are straightforward:
\eqref{sum to v} is clear;
\eqref{row sums of P} is already mentioned and is immediate from $I+A+\overline{A}=J$;
\eqref{linear traces}--\eqref{quadratic trace 3} are the values of $\operatorname{tr}(A)$, $\operatorname{tr}(\overline{A})$, $\operatorname{tr}(A^2)$, $\operatorname{tr}(A\overline{A})$, and $\operatorname{tr}(\smash{\overline{A}}\vphantom{A}^2)$;
\eqref{f and g} follows from \eqref{sum to v} and \eqref{linear traces}.
To show \eqref{fg}, first restrict \eqref{quadratic equation} to $\mathbb{U}_{\rvec}$ and $\mathbb{U}_{\svec}$ and use \eqref{row sums of P} to find that $r$ and $s$ are the solutions of the quadratic equation
\begin{equation*}
	\xi^2-(\lambda-\mu)\xi+\mu-k=0
\end{equation*}
in indeterminate $\xi$, so that we have
\begin{equation}\label{lamda,mu in terms of k,r,s}
		r+s=\lambda-\mu, \quad rs=\mu-k.
\end{equation}
Next, count the triples of distinct vertices $x,y,z$ such that $x\sim y\sim z\not\sim x$ (in $G$) in two ways to get
\begin{equation}\label{how parameters are related}
	k(k-1-\lambda)=\overline{k}\mu.
\end{equation}
Then use \eqref{f and g} together with \eqref{sum to v}, \eqref{lamda,mu in terms of k,r,s}, and \eqref{how parameters are related}.

\begin{lem}\label{half case}
If $r$ and $s$ are non-integral then $f=g$ and we have
\begin{equation}\label{half case parameters}
	v=4\ell+1, \quad k=2\ell, \quad \lambda=\ell-1, \quad \mu=\ell
\end{equation}
for some positive integer $\ell$.
Moreover, in this case we have
\begin{equation}\label{half case eigenvalues}
	r=\frac{-1+\sqrt{1+4\ell}}{2}, \quad s=\frac{-1-\sqrt{1+4\ell}}{2}.
\end{equation}
\end{lem}

\begin{proof}
See \cite[p.\,118]{BH2012B}.
We have $f=g$ since $r$ and $s$ are algebraic conjugates.
Using \eqref{f and g} and \eqref{lamda,mu in terms of k,r,s}, we then have $(v-1)(\mu-\lambda)=2k$.
Since $k<v-1$, this is possible only when $\mu-\lambda=1$ and $v-1=2k$.
In particular, we have $\overline{k}=k$ by \eqref{sum to v}, and hence it follows from \eqref{how parameters are related} that $k=2\mu$, as desired.
\end{proof}

Strongly regular graphs with parameters of the form \eqref{half case parameters} are called \emph{conference graphs}.

We say that $G$ is \emph{imprimitive} if either $G$ or $\overline{G}$ is disconnected, and \emph{primitive} otherwise.
Thus, $G$ is imprimitive if and only if $G=pK_q$ or $G=K_{p\times q}$ for some integers $p,q\geqslant 2$.

\begin{ex}[imprimitive graphs]
Let $p$ and $q$ be integers such that $p,q\geqslant 2$.
The disjoint union $pK_q$ is strongly regular with parameters $(pq,q-1,q-2,0)$ and restricted eigenvalues $r=q-1$, $s=-1$.
The complete multipartite graph $K_{p\times q}$ is strongly regular with parameters $(pq,(p-1)q,(p-2)q,(p-1)q)$ and restricted eigenvalues $r=0$, $s=-q$.
\end{ex}

We now introduce two more families of strongly regular graphs.
See \cite[Sections 9.1.10--9.1.13]{BH2012B}.
Recall that an \emph{incidence structure} is a triple $(\mathscr{P},\mathscr{B},\mathscr{I})$, where $\mathscr{P}$ and $\mathscr{B}$ are finite sets whose elements are called \emph{points} and \emph{blocks}, respectively, and where $\mathscr{I}\subset\mathscr{P}\times\mathscr{B}$.
If $(p,b)\in\mathscr{I}$ then we say that $p$ and $b$ are \emph{incident}, or $p$ is \emph{contained} in $b$, and so on.
The \emph{block graph} of $(\mathscr{P},\mathscr{B},\mathscr{I})$ is the graph $G=(V,E)$ with $V=\mathscr{B}$ where two distinct blocks are adjacent if and only if they contain a point in common.

\begin{ex}[Steiner graphs]
Let $m$ and $d$ be integers such that $2\leqslant m<d$.
A \emph{Steiner system} $S(2,m,d)$ is an incidence structure $(\mathscr{P},\mathscr{B},\mathscr{I})$ with $|\mathscr{P}|=d$ such that every block contains precisely $m$ points, and that any two distinct points are contained in a unique block.
The block graph of an $S(2,m,d)$ is called a \emph{Steiner graph} and is strongly regular with parameters $(v,k,\lambda,\mu)$ provided that $v>d$, where
\begin{equation*}
	v=\frac{d(d-1)}{m(m-1)}, \quad k=\frac{(d-m)m}{m-1}, \quad \lambda=(m-1)^2+\frac{d-2m+1}{m-1}, \quad \mu=m^2,
\end{equation*}
and with restricted eigenvalues
\begin{equation*}
	r=\frac{d-m^2}{m-1}, \quad s=-m.
\end{equation*}
\end{ex}

\begin{ex}[Latin square graphs]
Let $m$ and $e$ be integers with $m,e\geqslant 2$.
A \emph{transversal design} $\mathit{TD}(m,e)$ is an incidence structure $(\mathscr{P},\mathscr{B},\mathscr{I})$ where the point set is given a partition $\mathscr{P}=\mathscr{P}_1\sqcup\dots\sqcup\mathscr{P}_m$ into $m$ \emph{groups} of the same size $e$ (so $|\mathscr{P}|=me$), such that every block is incident with every group in exactly one point, and that any two points from distinct groups are contained in a unique block.
The block graph of a $\mathit{TD}(m,e)$ is called a \emph{Latin square graph} and is strongly regular with parameters $(v,k,\lambda,\mu)$ provided that $m\leqslant e$, where
\begin{equation*}
	v=e^2, \quad k=m(e-1), \quad \lambda=(m-1)(m-2)+e-2, \quad \mu=m(m-1),
\end{equation*}
and with restricted eigenvalues
\begin{equation*}
	r=e-m, \quad s=-m.
\end{equation*}
\end{ex}

The following fundamental result is due to Neumaier \cite{Neumaier1979AM}.

\begin{prop}[{\cite{Neumaier1979AM}}]\label{geometric}
For any fixed integer $m>0$, there are only finitely many primitive strongly regular graphs with least eigenvalue $s=-m$, other than Steiner graphs and Latin square graphs.
\end{prop}

\section{Proof of Theorem \ref{main theorem}}\label{sec: proof}

To prove Theorem \ref{main theorem}, we invoke L\'{e}vy's continuity theorem concerning the pointwise convergence of the characteristic functions; see e.g., \cite[Theorem 8.8.1]{Bogachev2007B}.
Thus, we fix $(\xi_1,\xi_2)\in\mathbb{R}^2$ throughout the proof.

Recall that $G$ depends on $n$ in general.
Let $\bm{A}$ and $\overline{\bm{A}}$ be the adjacency matrices of $G^{\square n}$ and $\myoverline{G}^{\square n}$, respectively.
For convenience, let
\begin{equation}\label{Lambda}
	\Lambda_n=\{(j,h):j,h=0,1,\dots,n,\,j+h\leqslant n\}.
\end{equation}
For every $(j,h)\in\Lambda_n$, consider the subspace
\begin{equation}\label{common eigenspace}
	\bigoplus_{\mathfrak{l}_1,\mathfrak{l}_2,\dots,\mathfrak{l}_n} \!\!\!\! \mathbb{U}_{\mathfrak{l}_1}\!\otimes \mathbb{U}_{\mathfrak{l}_2}\!\otimes\dots\otimes\mathbb{U}_{\mathfrak{l}_n}
\end{equation}
of $(\mathbb{C}^v)^{\otimes n}\cong\mathbb{C}^{v^n}$, where the sum is over $\mathfrak{l}_1,\mathfrak{l}_2,\dots,\mathfrak{l}_n\in\{\kvec,\rvec,\svec\}$ such that
\begin{equation*}
	\{\mathfrak{l}_1,\mathfrak{l}_2,\dots,\mathfrak{l}_n\}=\{\underbrace{\kvec,\dots,\kvec}_{n-j-h},\underbrace{\rvec,\dots,\rvec}_j,\underbrace{\svec,\dots,\svec}_h\,\}
\end{equation*}
as multisets.
It has dimension
\begin{equation*}
	\binom{n}{n-j-h,j,h}f^jg^h.
\end{equation*}
By virtue of \eqref{adjacency matrix of power} (and the corresponding formula for $\overline{\bm{A}}$), this subspace is a common eigenspace\footnote{It will turn out that the pairs $(\theta_{j,h},\overline{\theta}_{j,h})$ $((j,h)\in\Lambda_n)$ are mutually distinct, and that these subspaces are indeed the \emph{maximal} common eigenspaces of $(\bm{A},\overline{\bm{A}})$; see Section \ref{sec: Charlier-Hermite polynomials}. However, this fact is not necessary in the computation of \eqref{characteristic function} below.} of $(\bm{A},\overline{\bm{A}})$ with eigenvalues $(\theta_{j,h},\overline{\theta}_{j,h})$, where
\begin{equation}\label{eigenvalues}
	\theta_{j,h}=(n-j-h)k+jr+h s, \quad \overline{\theta}_{j,h}=(n-j-h)\overline{k}+j\overline{s}+h\overline{r}.
\end{equation}
Note that $G^{\square n}$ and $\myoverline{G}^{\square n}$ are $nk$-regular and $n\overline{k}$-regular, respectively.
Hence it follows from the above comments that the value of the characteristic function of $\nu_{G^{\square n},\,\myoverline{G}^{\square n}}$ at $(\xi_1,\xi_2)$ is given by
\begin{align}
	\varphi_{\operatorname{tr}}\!\left(\exp\!\left( \vphantom{\frac{\xi}{\sqrt{k}}} \right.\right. & \! \frac{i\xi_1 \bm{A}}{\sqrt{nk}} +\!\left.\left.\frac{i\xi_2 \overline{\bm{A}}}{\sqrt{n\overline{k}}} \right)\!\right) \label{characteristic function} \\ 
	&= \frac{1}{v^n}\!\sum_{(j,h)\in\Lambda_n} \!\! \exp\!\left(\frac{i\xi_1\theta_{j,h}}{\sqrt{nk}}+\frac{i\xi_2\overline{\theta}_{j,h}}{\sqrt{n\overline{k}}}\right)\!\!\binom{n}{n-j-h,j,h} f^j g^h \notag \\
	&= \frac{1}{v^n}\!\sum_{(j,h)\in\Lambda_n} \!\!\! \left(e^{\Delta_{\kvec}}\right)^{\! n-j-h} \!\left(f e^{\Delta_{\rvec}}\right)^j \!\left(g e^{\Delta_{\svec}}\right)^{\!h} \!\binom{n}{n-j-h,j,h} \notag \\
	&= \left( \frac{1}{v}e^{\Delta_{\kvec}}+\frac{f}{v}e^{\Delta_{\rvec}} + \frac{g}{v}e^{\Delta_{\svec}} \right)^{\!\!n} \notag \\
	&= \exp\!\left( n\log\!\left( \frac{1}{v}e^{\Delta_{\kvec}}+ \frac{f}{v}e^{\Delta_{\rvec}} + \frac{g}{v}e^{\Delta_{\svec}} \right) \!\right), \notag
\end{align}
where
\begin{equation}\label{Delta's}
	\Delta_{\kvec}=\frac{i\xi_1k}{\sqrt{nk}}+\frac{i\xi_2\overline{k}}{\sqrt{n\overline{k}}}, \quad \Delta_{\rvec}=\frac{i\xi_1r}{\sqrt{nk}}+\frac{i\xi_2\overline{s}}{\sqrt{n\overline{k}}}, \quad \Delta_{\svec}=\frac{i\xi_1s}{\sqrt{nk}}+\frac{i\xi_2\overline{r}}{\sqrt{n\overline{k}}}.
\end{equation}

Note by \eqref{row sums of P} that
\begin{equation}\label{limits of r-bar and s-bar}
	\frac{\overline{r}}{n}\rightarrow -\sigma, \quad \frac{\overline{s}}{n}\rightarrow -\rho,
\end{equation}
and by \eqref{range of eigenvalues} that
\begin{equation}\label{inequalities for the limits}
	-\min\{\kappa,\overline{\kappa}\}\leqslant \sigma\leqslant 0\leqslant \rho\leqslant\min\{\kappa,\overline{\kappa}\}.
\end{equation}

\subsection{The case \texorpdfstring{$\rho>0$ or $\sigma<0$}{rho>0 or sigma<0}}

First we consider the case where $\rho>0$ or $\sigma<0$.
Then we have $\kappa,\overline{\kappa}>0$ by \eqref{inequalities for the limits}, and moreover $1/v=O(1/n)$.
Note that each of $\Delta_{\kvec}$, $\Delta_{\rvec}$, and $\Delta_{\svec}$ converges by \eqref{limits of r-bar and s-bar}.
On the one hand, by \eqref{fg} we have
\begin{equation}\label{fg limit}
	\frac{fg}{n}\rightarrow \frac{\kappa\overline{\kappa} \omega}{(\rho-\sigma)^2} <\infty,
\end{equation}
so that
\begin{equation*}
	\frac{fg}{n^2}\rightarrow  0.
\end{equation*}
On the other hand, by \eqref{f and g} we have
\begin{equation*}
	\frac{f}{n}\rightarrow \frac{\omega\sigma}{\sigma-\rho}, \quad \frac{g}{n}\rightarrow \frac{\omega\rho}{\rho-\sigma}.
\end{equation*}
Hence we have $\rho=0$ or $\sigma=0$.

For the moment, assume that $\rho=0$ and $\sigma<0$, so that
\begin{equation*}
	\frac{f}{n}\rightarrow \omega, \quad \frac{g}{n}\rightarrow 0.
\end{equation*}
Then it follows from \eqref{fg limit} that
\begin{equation*}
	g\rightarrow g_{\infty}:= \frac{\kappa\overline{\kappa}}{\sigma^2}.
\end{equation*}
In particular, $g$ is bounded.
Moreover, by \eqref{linear traces} we have
\begin{equation*}
	r=-\frac{k+sg}{f}\rightarrow r_{\infty}:=-\frac{\kappa+\sigma g_{\infty}}{\omega},
\end{equation*}
so that $r$ and $\overline{s}=-r-1$ are also bounded, and thus $\Delta_{\rvec}=O(1/n)$.
Since by \eqref{sum to v}
\begin{equation*}
	\frac{f}{v}e^{\Delta_{\rvec}} = e^{\Delta_{\rvec}} -\frac{1+g}{v}e^{\Delta_{\rvec}} = 1+\Delta_{\rvec}-\frac{1+g}{v}+O\!\left(\frac{1}{n^2}\right)\!,
\end{equation*}
it follows using \eqref{limits of r-bar and s-bar} that \eqref{characteristic function} equals
\begin{align}
	\exp & \!\left( n\!\left( \frac{1}{v}e^{\Delta_{\kvec}} + \frac{g}{v}e^{\Delta_{\svec}} +\Delta_{\rvec} -\frac{1+g}{v}  +O\!\left(\frac{1}{n^2} \right)\! \right)\!\right) \label{Poisson} \\
	&= \exp\!\left( \frac{n}{v}e^{\Delta_{\kvec}} + \frac{ng}{v}e^{\Delta_{\svec}} +n\Delta_{\rvec} -\frac{n(1+g)}{v} +O\!\left(\frac{1}{n} \right)\! \right) \notag \\
	&\rightarrow \exp\!\left( \exp\!\left(i\xi_1\sqrt{\kappa}+i\xi_2\sqrt{\overline{\kappa}}\,\right)\!\frac{1}{\omega} + \exp\!\left(\frac{i\xi_1\sigma}{\sqrt{\kappa}}-\frac{i\xi_2\sigma}{\sqrt{\overline{\kappa}}}\right)\!\frac{g_{\infty}}{\omega} \right. \notag \\
	& \qquad\qquad\qquad\qquad \left. +\frac{i\xi_1r_{\infty}}{\sqrt{\kappa}}-\frac{i\xi_2(r_{\infty}+1)}{\sqrt{\overline{\kappa}}} -\frac{1+g_{\infty}}{\omega} \right). \notag
\end{align}
We note that the limit in \eqref{Poisson} is the value at $(\xi_1,\xi_2)$ of the characteristic function of an affine transformation of a bivariate Poisson distribution.

We now show that $G$ is a complete multipartite graph for $n\gg 0$, so that we have $\sigma=-\overline{\kappa}$, $r_{\infty}=0$, $g_{\infty}=\kappa/\overline{\kappa}$, and the limit in \eqref{Poisson} becomes
\begin{equation*}
	\exp\!\left( \exp\!\left(i\xi_1\sqrt{\kappa}+i\xi_2\sqrt{\overline{\kappa}}\,\right)\!\frac{1}{\omega} + \exp\!\left(-\frac{i\xi_1\overline{\kappa}}{\sqrt{\kappa}}+i\xi_2\sqrt{\overline{\kappa}}\right)\!\frac{\kappa}{\omega\overline{\kappa}} -\frac{i\xi_2}{\sqrt{\overline{\kappa}}} -\frac{1}{\overline{\kappa}} \right),
\end{equation*}
which corresponds to the distribution $\nu$ given in Theorem \ref{main theorem}\,(i).
Recall that $\overline{s}$ is bounded.
By virtue of Proposition \ref{geometric} and Lemma \ref{half case}, $\overline{G}$ is one of the following for $n\gg 0$: ($\overline{G}_1$) a conference graph; ($\overline{G}_2$) a disjoint union $pK_q$ of complete graphs; ($\overline{G}_3$) a complete multipartite graph $K_{p\times q}$; ($\overline{G}_4$) a Steiner graph of an $S(2,m,d)$; ($\overline{G}_5$) a Latin square graph of a $\mathit{TD}(m,e)$.
Case ($\overline{G}_1$) is impossible as $v$ would also be bounded.
For Case ($\overline{G}_3$), we have $\sigma=0$, a contradiction.
If $\overline{G}$ is a Steiner graph of an $S(2,m,d)$ as in Case ($\overline{G}_4$), then $m$ is bounded since $\overline{s}=-m$.
However, since $\overline{k}$ and $v$ are linear and quadratic in $d$, respectively, $\overline{\kappa}$ and $\omega$ cannot be both finite and non-zero, a contradiction.
The same argument shows that Case ($\overline{G}_5$) is also impossible.
Hence we are left with Case ($\overline{G}_2$), so that we have (i) in Theorem \ref{main theorem}.

If $\rho>0$ and $\sigma=0$, then switching the roles of $G$ and $\overline{G}$ gives (ii) in Theorem \ref{main theorem}.
This completes the case where $\rho>0$ or $\sigma<0$.

\subsection{The case \texorpdfstring{$\rho=\sigma=0$}{rho=sigma=0}}
\label{sec: rho=sigma=0}

Next we deal with the case where $\rho=\sigma=0$.
Note that $\Delta_{\kvec}$ converges, and that $\Delta_{\rvec},\Delta_{\svec}\rightarrow 0$ in view of \eqref{range of eigenvalues} and \eqref{limits of r-bar and s-bar}.
From \eqref{quadratic trace 1}, \eqref{quadratic trace 2}, and \eqref{quadratic trace 3} it follows that
\begin{equation*}
	r^2f\leqslant kv, \quad -r\overline{s}f\leqslant k\overline{k}, \quad \myoverline{s}^2f\leqslant \overline{k}v,
\end{equation*}
from which it follows that
\begin{equation}\label{estimation}
	\bigl|\Delta_{\rvec}^2\bigr|f \leqslant \left(\frac{\xi_1^2 r^2}{nk}-2\frac{|\xi_1\xi_2|r\overline{s}}{n\sqrt{k\overline{k}}}+\frac{\xi_2^2\myoverline{s}^2}{n\overline{k}}\right)\!f  \leqslant \frac{\xi_1^2v}{n}+2\frac{|\xi_1\xi_2|\sqrt{k\overline{k}}}{n}+\frac{\xi_2^2v}{n}.
\end{equation}
Hence $\Delta_{\rvec}^2f$ is bounded.
Likewise, we can show that $\Delta_{\svec}^2g$ is bounded.
We also need the following identities which are immediate from \eqref{linear traces}--\eqref{quadratic trace 3}:
\begin{gather}
	\Delta_{\kvec}+\Delta_{\rvec}f+\Delta_{\svec}g=0, \label{linear in Delta} \\
	\Delta_{\kvec}^2+\Delta_{\rvec}^2f+\Delta_{\svec}^2g=-\frac{\xi_1^2v}{n}-\frac{\xi_2^2v}{n}. \label{quadratic in Delta}
\end{gather}

For the moment, assume that $\kappa>0$ or $\overline{\kappa}>0$.
Note that $1/v=O(1/n)$ in this case.
Moreover, we have
\begin{equation*}
	\frac{f}{v}e^{\Delta_{\rvec}}=\left(1+\Delta_{\rvec}+\frac{\Delta_{\rvec}^2}{2}+O\!\left(\Delta_{\rvec}^3\right)\!\right)\!\frac{f}{v}=\left(1+\Delta_{\rvec}+\frac{\Delta_{\rvec}^2}{2}\right)\!\frac{f}{v}+o\!\left(\frac{1}{n}\right),
\end{equation*}
and similarly for $ge^{\Delta_{\svec}}/v$.
Hence it follows from \eqref{sum to v}, \eqref{linear in Delta}, and \eqref{quadratic in Delta} that \eqref{characteristic function} equals
\begin{align}
	\exp & \!\left( n\log\!\left( \frac{1}{v}e^{\Delta_{\kvec}}+\left(1+\Delta_{\rvec}+\frac{\Delta_{\rvec}^2}{2}\right)\!\frac{f}{v}\right.\right. \label{similar computations will be done later} \\
	& \qquad\qquad\qquad\qquad +\left.\left. \left(1+\Delta_{\svec}+\frac{\Delta_{\svec}^2}{2}\right)\!\frac{g}{v}+o\!\left(\frac{1}{n}\right)\! \right)\!\right) \notag \\
	&= \exp\!\left( n\log\!\left( 1+ \frac{1}{v}e^{\Delta_{\kvec}} -\frac{\Delta_{\kvec}}{v} \right.\right. \notag \\
	&  \qquad\qquad\qquad\qquad \left.\left. +\Bigl(\Delta_{\rvec}^2f+\Delta_{\svec}^2g\Bigr)\frac{1}{2v}-\frac{1}{v}+o\!\left(\frac{1}{n}\right)\!\right)\!\right) \notag \\
	&= \exp\!\left(n\!\left( \frac{1}{v}e^{\Delta_{\kvec}} -\frac{\Delta_{\kvec}}{v}+\Bigl(\Delta_{\rvec}^2f+\Delta_{\svec}^2g\Bigr)\frac{1}{2v}-\frac{1}{v}+o\!\left(\frac{1}{n}\right)\!\right)\!\right) \notag \\
	&\rightarrow \exp\!\left(\exp\!\left(i\xi_1\sqrt{\kappa}+i\xi_2\sqrt{\overline{\kappa}}\,\right)\!\frac{1}{\omega}-\frac{i\xi_1\sqrt{\kappa}+i\xi_2\sqrt{\overline{\kappa}}}{\omega} \right. \notag \\
	& \qquad\qquad\qquad\qquad \left. -\frac{\bigl(\xi_1\sqrt{\overline{\kappa}}-\xi_2\sqrt{\kappa}\,\bigr)^{\!2}}{2\omega}-\frac{1}{\omega}\right). \notag
\end{align}
It is a straightforward matter to show that this corresponds to the distribution $\nu$ in Theorem \ref{main theorem}\,(iii).

Finally, assume that $\kappa=\overline{\kappa}=0$.
In this case, we have $\omega=0$, i.e.,
\begin{equation*}
	\frac{v}{n}\rightarrow 0.
\end{equation*}
Note that $\Delta_{\kvec},\Delta_{\rvec},\Delta_{\svec}=O((v/n)^{1/2})$.
Moreover, it follows from \eqref{estimation} that $\Delta_{\rvec}^2f=O(v/n)$, and likewise we have $\Delta_{\svec}^2g=O(v/n)$.
Hence it follows from \eqref{sum to v}, \eqref{linear in Delta}, and \eqref{quadratic in Delta} that \eqref{characteristic function} equals
\begin{align*}
	\exp & \!\left( n\log\!\left( \! \left(1+\Delta_{\kvec}+\frac{\Delta_{\kvec}^2}{2}\right)\!\frac{1}{v}+\left(1+\Delta_{\rvec}+\frac{\Delta_{\rvec}^2}{2}\right)\!\frac{f}{v}\right.\right. \\
	& \qquad\qquad\qquad\qquad +\left.\left. \left(1+\Delta_{\svec}+\frac{\Delta_{\svec}^2}{2}\right)\!\frac{g}{v}+O\!\left(\!\left(\frac{v}{n}\right)^{\!\frac{3}{2}}\right)\!\frac{1}{v} \right)\!\right) \\
	&= \exp\!\left( n\log\!\left( \! 1-\frac{\xi_1^2}{2n}-\frac{\xi_2^2}{2n}+O\!\left(\!\left(\frac{v}{n}\right)^{\!\frac{3}{2}}\right)\!\frac{1}{v} \right)\!\right) \\
	&= \exp\!\left(n\!\left(\!-\frac{\xi_1^2}{2n}-\frac{\xi_2^2}{2n}+O\!\left(\!\left(\frac{v}{n}\right)^{\!\frac{3}{2}}\right)\!\frac{1}{v}\right)\!\right) \\
	&= \exp\!\left(-\frac{\xi_1^2}{2}-\frac{\xi_2^2}{2}+O\!\left(\!\left(\frac{v}{n}\right)^{\!\frac{1}{2}}\right)\!\right) \\
	&\rightarrow \exp\!\left(-\frac{\xi_1^2}{2}-\frac{\xi_2^2}{2}\right).
\end{align*}
This corresponds to the standard bivariate Gaussian distribution, and hence we have (iv) in Theorem \ref{main theorem}.

This completes the proof of Theorem \ref{main theorem}.

\section{Examples}\label{sec: examples}

The graph $G$ (which we recall is a function in $n$) is already identified for (i) and (ii) in Theorem \ref{main theorem}, whereas (iv) is a degenerate case and is easily realized as it only requires $v/n\rightarrow 0$.
Below are some examples for (iii) in Theorem \ref{main theorem}, i.e., such that $\kappa>0$ or $\overline{\kappa}>0$, and $\rho=\sigma=0$.

\begin{ex}
Consider the imprimitive strongly regular graphs $pK_q$ and $K_{p\times q}$.
Assume that $pq$ is (essentially) linear in $n$ and that $q/n\rightarrow 0$.
Then we have $\kappa=0$ and $\overline{\kappa}>0$ in Theorem \ref{main theorem}\,(iii) for $pK_q$, and $\kappa>0$ and $\overline{\kappa}=0$ for $K_{p\times q}$.
\end{ex}

\begin{ex}[Paley graphs]
Let $q$ be a prime power with $q\equiv 1\,(\operatorname{mod}4)$.
The \emph{Paley graph} $\operatorname{Paley}(q)$ has vertex set $\mathbb{F}_q$ (the finite field with $q$ elements), where two distinct vertices are adjacent if and only if their difference is a square.
It is easy to see that $\operatorname{Paley}(q)$ is a conference graph.
See \cite[Sections 9.1.1 and 9.1.2]{BH2012B}.
Hence if we take $n$ to be linear in $q$, then it follows from \eqref{half case parameters} and \eqref{half case eigenvalues} that we are in Theorem \ref{main theorem}\,(iii) with $\kappa=\overline{\kappa}>0$.
\end{ex}

\begin{ex}[Symplectic graphs]
Let $q$ be a prime power, and let $\ell\geqslant 2$ be an integer.
We endow $\mathbb{F}_q^{2\ell}$ with a non-degenerate symplectic form.
The \emph{Symplectic graph} $\mathit{Sp}_{2\ell}(q)$ has as vertex set the set of one-dimensional subspaces (i.e., projective points) of $\mathbb{F}_q^{2\ell}$, where two distinct vertices are adjacent if and only if they are orthogonal.
The graph $\mathit{Sp}_{2\ell}(q)$ is strongly regular with parameters $(v,k,\lambda,\mu)$, where
\begin{equation*}
	v=\frac{q^{2\ell}-1}{q-1}, \quad k=\frac{q^{2\ell-1}-q}{q-1}, \quad \lambda=\frac{q^{2\ell-2}-2q+1}{q-1},  \quad \mu=\frac{q^{2\ell-2}-1}{q-1},
\end{equation*}
and with restricted eigenvalues
\begin{equation*}
	r=q^{\ell-1}-1, \quad s=-q^{\ell-1}-1.
\end{equation*}
Fix $q$ and let $\ell\rightarrow\infty$.
If $n$ is linear in $q^{2\ell}$ then again we are in Theorem \ref{main theorem}\,(iii) with $\kappa=\overline{\kappa}/(q-1)>0$.
There are many other infinite families of strongly regular graphs related to finite geometry; see \cite{BrouwerWWW} and \cite[Section 9.9]{BH2012B}.
\end{ex}

\begin{ex}\label{many TDs}
Let $q$ be a prime power.
Let $H_1,H_2,\dots,H_m$ be distinct one-dimensional subspaces of $\mathbb{F}_q^2$,
where $1\leqslant m\leqslant q$.
For $j=1,2,\dots,m$, let $\mathscr{P}_j$ be the set of $q$ parallel affine subspaces of $\mathbb{F}_q^2$ with direction $H_j$, i.e.,
\begin{equation*}
	\mathscr{P}_j=\{H_j+x:x\in\mathbb{F}_q^2\} \quad (j=1,2,\dots,m).
\end{equation*}
Let $\mathscr{P}=\mathscr{P}_1\sqcup\dots\sqcup\mathscr{P}_m$ and $\mathscr{B}=\mathbb{F}_q^2$.
Consider the incidence structure $(\mathscr{P},\mathscr{B},\mathscr{I})$, where a point $H_j+x$ and a block $y$ are incident if and only if $y\in H_j+x$.
Then it is easy to see that $(\mathscr{P},\mathscr{B},\mathscr{I})$ is a $\mathit{TD}(m,q)$.
Hence if we take both $q^2$ and $mq$ to be linear in $n$, then the corresponding Latin square graph attains Theorem \ref{main theorem}\,(iii) with $\kappa>0$.
We may view $\operatorname{Paley}(q^2)$ (for odd $q$) in this way with $m=(q+1)/2$, as $\mathbb{F}_{q^2}\cong\mathbb{F}_q^2$.
We note that, unlike the previous examples, \emph{any} $\kappa,\overline{\kappa}\geqslant 0$ can be achieved here as limits.
There is also a more general construction of strongly regular graphs from cyclotomy, all giving rise to examples of Theorem \ref{main theorem}\,(iii); cf.~\cite[Section 9.8.5]{BH2012B}.
\end{ex}

\section{Bivariate Charlier--Hermite polynomials}\label{sec: Charlier-Hermite polynomials}

It is more natural to understand our approach of considering the pair $(G^{\square n},\myoverline{G}^{\square n})$ with $G$ and $\overline{G}$ both strongly regular, in terms of \emph{association schemes}; cf.~\cite[Chapter 11]{BH2012B}.
An \emph{association scheme with $d$ classes} is an edge decomposition of a complete graph $K_v$ into $d$ graphs $G_1,G_2,\dots,G_d$, whose adjacency matrices $A_1,A_2,\dots,A_d$, together with $A_0:=I$, form a basis of a $(d+1)$-dimensional matrix $*$-algebra $M$ over $\mathbb{C}$, called the \emph{Bose--Mesner algebra}.
Thus, the association schemes with one class are the same thing as the complete graphs, and those with two classes are the pairs of strongly regular graphs and their complements (cf.~\eqref{quadratic equation}).
Since $M$  consists of symmetric matrices only, it follows that $M$ is commutative and has precisely $d+1$ maximal common eigenspaces.

With the notation in the previous sections, we now consider the association scheme $G,\overline{G}$ with Bose--Mesner algebra $M=\langle I,A,\overline{A}\rangle$.
Observe that the $n^{\mathrm{th}}$ symmetric tensor space $\operatorname{Sym}^n(M)$ of $M$ is again the Bose--Mesner algebra of another association scheme consisting of the graphs on $V^n$ with adjacency matrices (cf.~\eqref{Lambda})
\begin{equation}\label{adjacency matrices of extension}
	\bm{A}_{a,b}=\!\sum_{B_1,B_2,\dots,B_n} \!\!\!\! B_1\otimes B_2 \otimes \dots \otimes B_n \quad ((a,b)\in\Lambda_n),
\end{equation}
where the sum is over $B_1,B_2,\dots,B_n\in\{I,A,\overline{A}\}$ such that
\begin{equation*}
	\{B_1,B_2,\dots,B_n\}=\{\underbrace{I,\dots,I}_{n-a-b},\underbrace{A,\dots,A}_a,\underbrace{\overline{A},\dots,\overline{A}}_b\,\}
\end{equation*}
as multisets; cf.~\cite[Section 2.5]{Delsarte1973PRRS}.
In particular, we have (cf.~\eqref{adjacency matrix of power})
\begin{equation*}
	\bm{A}=\bm{A}_{1,0}, \quad \overline{\bm{A}}=\bm{A}_{0,1}.
\end{equation*}
Moreover, the subspaces \eqref{common eigenspace} for $(j,h)\in\Lambda_n$ are the maximal common eigenspaces of $\operatorname{Sym}^n(M)$.
Mizukawa and Tanaka \cite{MT2004PAMS} showed that the corresponding eigenvalues of $\bm{A}_{a,b}$ are expressed as
\begin{equation*}
	k_{a,b}\hspace{.01in} \mathcal{K}_{a,b}(j,h) \quad ((j,h)\in\Lambda_n),
\end{equation*}
where
\begin{equation}\label{degree}
	k_{a,b}=\binom{n}{n-a-b,a,b} k^a\myoverline{k}^b
\end{equation}
denotes the degree of the graph with adjacency matrix $\bm{A}_{a,b}$, and $\mathcal{K}_{a,b}(j,h)$ is the terminating Aomoto--Gelfand hypergeometric series
\begin{equation*}
	\mathcal{K}_{a,b}(j,h)=\sum_{ \ell_1,\ell_2,\ell_3,\ell_4} \!\!\! \frac{ (-a)_{\ell_1+\ell_3} (-b)_{\ell_2+\ell_4} (-j)_{\ell_1+\ell_2} (-h)_{\ell_3+\ell_4} }{(-n)_{\ell_1+\ell_2+\ell_3+\ell_4} \ell_1! \ell_2! \ell_3! \ell_4! } u_1^{\ell_1} u_2^{\ell_2} u_3^{\ell_3} u_4^{\ell_4},
\end{equation*}
where the sum is over $\ell_1,\ell_2,\ell_3,\ell_4=0,1,\dots,n$ with $\ell_1+\ell_2+\ell_3+\ell_4\leqslant n$, and
\begin{equation*}
	u_1= 1-\frac{r}{k}, \quad u_2=1-\frac{\overline{s}}{\overline{k}}, \quad u_3=1-\frac{s}{k}, \quad u_4=1-\frac{\overline{r}}{\overline{k}}.
\end{equation*}
We note that\footnote{This in particular shows that $\operatorname{Sym}^n(M)$ is generated by $\bm{A}$ and $\overline{\bm{A}}$ which correspond to linear polynomials, so that the subspaces \eqref{common eigenspace} are also the maximal common eigenspaces of $(\bm{A},\overline{\bm{A}})$.} $\mathcal{K}_{a,b}(j,h)$ is a polynomial in $j$ and $h$ with (total) degree $a+b$.
The $\mathcal{K}_{a,b}$ are the \emph{bivariate Krawtchouk polynomials} and are also known as the \emph{Rahman polynomials}; cf.~\cite{Griffiths1971AJS,Grunbaum2007SIGMA,HR2008SIGMA,IT2012TAMS}.
For $(j,h)\in\Lambda_n$, we have (cf.~\cite{Godsil2010M,Griffiths1971AJS,Tarnanen1987P})
\begin{align}
	\sum_{(a,b)\in\Lambda_n} \!\!\! k_{a,b}\hspace{.01in} & \mathcal{K}_{a,b}(j,h)\hspace{.01in} \xi_1^a\xi_2^b \label{generating function} \\
	& =(1+k\xi_1+\overline{k}\xi_2)^{n-j-h} (1+r\xi_1+\overline{s}\xi_2)^j (1+s\xi_1+\overline{r}\xi_2)^h, \notag
\end{align}
where $\xi_1$ and $\xi_2$ are indeterminates.
It should be remarked that, if $j$ and $h$ are also indeterminates, then the RHS above belongs to the formal power series ring $\mathbb{C}[j,h][[\xi_1,\xi_2]]$ over the polynomial ring $\mathbb{C}[j,h]$, and the coefficients of $\xi_1^a\xi_2^b$ on both sides still agree whenever $a+b\leqslant n$.

The goal of this section is to construct a family of bivariate hypergeometric orthogonal polynomials with respect to the probability measure $\nu$ in Theorem \ref{main theorem}\,(iii) as limits of the $\mathcal{K}_{a,b}$, so that we will assume from now on that
\begin{equation*}
	\omega=\kappa+\overline{\kappa}>0, \quad \rho=\sigma=0.
\end{equation*}
(For the other cases in Theorem \ref{main theorem}, similar discussions give rise to the bivariate Charlier and Hermite polynomials; see Remark \ref{OPs for other cases} below.)
For the convergence, however, we will work with the normalization $\sqrt{\vphantom{k}\smash{k_{a,b}}}\, \mathcal{K}_{a,b}$ which gives the eigenvalues of $\bm{A}_{a,b}/\sqrt{\vphantom{k}\smash{k_{a,b}}}$.
Moreover, we also make the change of variables (cf.~\eqref{eigenvalues})
\begin{equation}\label{x,y}
	x=\frac{\theta_{j,h}}{\sqrt{nk}}, \quad y=\frac{\overline{\theta}_{j,h}}{\sqrt{n\overline{k}}},
\end{equation}
or equivalently,
\begin{align}
	j&=\frac{(\overline{k}-\overline{r})\sqrt{nk}}{v(r-s)}x - \frac{(k-s)\sqrt{n\overline{k}}}{v(r-s)}y +\frac{nf}{v}, \label{j in x,y} \\
	h&=\frac{(\overline{k}-\overline{s})\sqrt{nk}}{v(s-r)}x - \frac{(k-r)\sqrt{n\overline{k}}}{v(s-r)}y +\frac{ng}{v},\label{h in x,y}
\end{align}
where we have used \eqref{sum to v}, \eqref{row sums of P}, and \eqref{f and g}.
Thus, we define the bivariate polynomials $\hat{\mathcal{K}}_{a,b}(x,y)$ by
\begin{equation*}
	\hat{\mathcal{K}}_{a,b}(x,y)=\sqrt{\vphantom{k}\smash{k_{a,b}}}\,\mathcal{K}_{a,b}(j,h),
\end{equation*}
where the variables $(j,h)$ and $(x,y)$ are related as above.
Then, the $\hat{\mathcal{K}}_{a,b}$ satisfy the following orthogonality relation:
\begin{equation}\label{orthogonality for Krawtchouk}
	\int_{\mathbb{R}^2} \! \hat{\mathcal{K}}_{a,b}(x,y)\hspace{.01in}\hat{\mathcal{K}}_{c,d}(x,y)\,\nu_{G^{\square n},\,\myoverline{G}^{\square n}}(dxdy)=\varphi_{\operatorname{tr}}\!\left(\frac{\bm{A}_{a,b}\bm{A}_{c,d}}{\sqrt{k_{a,b}k_{c,d}}}\right)=\delta_{a,c}\delta_{b,d}
\end{equation}
for $(a,b),(c,d)\in\Lambda_n$.
We also note that
\begin{equation}\label{linear K polynomials}
	\hat{\mathcal{K}}_{1,0}(x,y)=x, \quad \hat{\mathcal{K}}_{0,1}(x,y)=y.
\end{equation}

Recall $\Delta_{\kvec},\Delta_{\rvec}$, and $\Delta_{\svec}$ from \eqref{Delta's}, where $\xi_1$ and $\xi_2$ are indeterminates in the present context, rather than real scalars.
To compute the limits of the $\hat{\mathcal{K}}_{a,b}$, we consider instead of \eqref{generating function}
\begin{equation}\label{normalized generating function}
	(1-i\Delta_{\kvec})^{n-j-h} (1-i\Delta_{\rvec})^j (1-i\Delta_{\svec})^h,
\end{equation}
which is an element of $\mathbb{C}[x,y][[\xi_1,\xi_2]]$ via \eqref{j in x,y} and \eqref{h in x,y}, where we view $x$ and $y$ as indeterminates.
On the one hand, in view of the remark after \eqref{generating function}, the terms in \eqref{normalized generating function} of degree at most $n$ (in $\xi_1$ and $\xi_2$) are
\begin{equation}
	\frac{k_{a,b}\hspace{.01in} \mathcal{K}_{a,b}(j,h)}{\sqrt{n^{a+b}k^a\myoverline{k}^b}}\xi_1^a\xi_2^b=\sqrt{\dfrac{k_{a,b}}{n^{a+b}k^a\myoverline{k}^b}}\,\hat{\mathcal{K}}_{a,b}(x,y)\hspace{.01in} \xi_1^a\xi_2^b \quad ((a,b)\in\Lambda_n),
\end{equation}
where by \eqref{degree} we have
\begin{equation}
	\frac{k_{a,b}}{n^{a+b}k^a\myoverline{k}^b} \rightarrow \frac{1}{a!b!}
\end{equation}
for fixed $a,b=0,1,2,\dots$.
On the other hand, since
\begin{equation*}
	n-j-h=\frac{\sqrt{nk}\,x+\!\sqrt{n\overline{k}}\,y+n}{v}
\end{equation*}
by \eqref{sum to v} and \eqref{row sums of P}, we have
\begin{equation}\label{first part converges}
	(1-i\Delta_{\kvec})^{n-j-h} \rightarrow \left(1+\!\sqrt{\kappa}\,\xi_1+\!\sqrt{\overline{\kappa}}\,\xi_2\right)^{\!\left(\!\sqrt{\kappa}\,x+\sqrt{\overline{\kappa}}\,y+1\right)/\omega}
\end{equation}
as elements of $\mathbb{C}[x,y][[\xi_1,\xi_2]]$.
The limit of the latter two factors of \eqref{normalized generating function} can be computed in the same manner as in the proof of Theorem \ref{main theorem}\,(iii); cf.~Section \ref{sec: rho=sigma=0}.
Recall that $\Delta_{\rvec},\Delta_{\svec}\rightarrow 0$.
Since the coefficients of $f\Delta_{\rvec}^2$ and $g\Delta_{\svec}^2$ are bounded as in \eqref{estimation}, we have
\begin{equation*}
	f\Delta_{\rvec}^{\ell},\,g\Delta_{\svec}^{\ell}\rightarrow 0 \quad (\ell=3,4,\dots).
\end{equation*}
Likewise, using \eqref{range of eigenvalues} and \eqref{row sums of P}, we can routinely show that
\begin{equation*}
	\biggl(j-\frac{nf}{v}\biggr)\Delta_{\rvec}^{\ell}, \biggl(h-\frac{ng}{v}\biggr)\Delta_{\svec}^{\ell} \rightarrow 0 \quad (\ell=2,3,\dots).
\end{equation*}
Since $\Delta_{\rvec}^{\ell},\Delta_{\svec}^{\ell}$ $(\ell=1,2,\dots)$ are homogeneous of degree $\ell$ in $\xi_1$ and $\xi_2$, it follows from \eqref{row sums of P}, \eqref{linear in Delta}, and \eqref{quadratic in Delta} that (cf.~\eqref{similar computations will be done later})
\begin{align}
	(1-i\Delta_{\rvec})^j & (1-i\Delta_{\svec})^h \label{latter part converges} \\
	&= \exp\!\left(-j\sum_{\ell=1}^{\infty}\frac{(i\Delta_{\rvec})^{\ell}}{\ell} -h\sum_{\ell=1}^{\infty}\frac{(i\Delta_{\svec})^{\ell}}{\ell}\right) \notag \\
	&= \exp\!\left( -\frac{nf}{v}\!\left(\!(i\Delta_{\rvec})+\frac{(i\Delta_{\rvec})^2}{2}\right)\! -\frac{ng}{v}\!\left(\!(i\Delta_{\svec})+\frac{(i\Delta_{\svec})^2}{2}\right) \right. \notag \\
	& \qquad\qquad\qquad \left. -\biggl(j-\frac{nf}{v}\biggr)(i\Delta_{\rvec})-\biggl(h-\frac{ng}{v}\biggr)(i\Delta_{\svec}) +o(1) \right) \notag \\
	&= \exp\!\left( \frac{n(i\Delta_{\kvec})}{v} -\frac{n}{2v}\Bigl(f(i\Delta_{\rvec})^2+g(i\Delta_{\svec})^2\Bigr) \right. \notag \\
	& \qquad\qquad\qquad \left. + \frac{\bigl(\sqrt{\overline{k}}\,x-\sqrt{k}\,y\bigr)\bigl(\sqrt{\overline{k}}\,\xi_1-\sqrt{k}\,\xi_2\bigr)}{v} +o(1) \right) \notag \\
	&\rightarrow \exp\!\left( -\frac{\sqrt{\kappa}\,\xi_1+\sqrt{\overline{\kappa}}\,\xi_2}{\omega} -\frac{\bigl(\sqrt{\overline{\kappa}}\,\xi_1-\sqrt{\kappa}\,\xi_2\bigr)^{\!2}}{2\omega} \right. \notag \\
	& \qquad\qquad\qquad \left. +\frac{\bigl(\sqrt{\overline{\kappa}}\,x-\sqrt{\kappa}\,y\bigr)\bigl(\sqrt{\overline{\kappa}}\,\xi_1-\sqrt{\kappa}\,\xi_2\bigr)}{\omega} \right). \notag
\end{align}
From \eqref{normalized generating function}--\eqref{latter part converges} it follows that there are polynomials $\mathcal{CH}_{a,b}(x,y)$ $(a,b=0,1,2,\dots)$ in $x$ and $y$ such that
\begin{equation*}
	\hat{\mathcal{K}}_{a,b}(x,y) \rightarrow \mathcal{CH}_{a,b}(x,y) \quad (a,b=0,1,2,\dots),
\end{equation*}
and that the generating function
\begin{equation*}
	\sum_{a,b=0}^{\infty}\frac{\mathcal{CH}_{a,b}(x,y)}{\sqrt{a!b!}}\xi_1^a\xi_2^b
\end{equation*}
equals the product of the limits in \eqref{first part converges} and \eqref{latter part converges} as elements in $\mathbb{C}[x,y][[\xi_1,\xi_2]]$.
We call the $\mathcal{CH}_{a,b}$ the \emph{bivariate Charlier--Hermite polynomials}.
We note by \eqref{linear K polynomials} that
\begin{equation}\label{linear CH polynomials}
	\mathcal{CH}_{1,0}(x,y)=x, \quad \mathcal{CH}_{0,1}(x,y)=y.
\end{equation}
The $\mathcal{CH}_{a,b}$ depend on two parameters $\kappa,\overline{\kappa}\geqslant 0$, where $\omega=\kappa+\overline{\kappa}>0$.

From the generating function we may derive an explicit formula for the $\mathcal{CH}_{a,b}$.
Observe that the limit in \eqref{first part converges} equals
\begin{equation*}
	\sum_{\ell_1,\ell_2=0}^{\infty} \!\! \frac{\bigl(-\frac{\sqrt{\kappa}\,x+\sqrt{\overline{\kappa}}\,y+1}{\omega}\bigr)_{\ell_1+\ell_2}}{\ell_1!\ell_2!} (-1)^{\ell_1+\ell_2} (\sqrt{\kappa}\,\xi_1)^{\ell_1} (\sqrt{\overline{\kappa}}\,\xi_2)^{\ell_2},
\end{equation*}
whereas the limit in \eqref{latter part converges} equals the product of the following two elements:
\begin{gather*}
	\sum_{\ell_3,\ell_4,\ell_5=0}^{\infty}\frac{1}{\ell_3!\ell_4!\ell_5!}\!\left(\!-\frac{\overline{\kappa}\xi_1^2}{2\omega}\right)^{\!\!\ell_3}\!\!\left(\!-\frac{\kappa \xi_2^2}{2\omega}\right)^{\!\!\ell_4}\!\!\left(\!\frac{\sqrt{\kappa\overline{\kappa}}\,\xi_1\xi_2}{\omega}\right)^{\!\!\ell_5}, \\
	\sum_{\ell_6,\ell_7=0}^{\infty} \frac{1}{\ell_6!\ell_7!} \!\left(\frac{\overline{\kappa}x-\!\sqrt{\kappa\overline{\kappa}}\,y-\!\sqrt{\kappa}}{\omega}\,\xi_1\right)^{\!\!\ell_6}\!\!\left(\frac{\kappa y-\!\sqrt{\kappa\overline{\kappa}}\,x-\!\sqrt{\overline{\kappa}}}{\omega}\,\xi_2\right)^{\!\!\ell_7}.
\end{gather*}
Picking out the coefficient of $\xi_1^a\xi_2^b$ in the product of the above three elements and simplifying the result using
\begin{equation*}
	\ell_6=a-\ell_1-2\ell_3-\ell_5, \quad \ell_7=b-\ell_2-2\ell_4-\ell_5,
\end{equation*}
it follows that $\mathcal{CH}_{a,b}(x,y)$ equals
\begin{equation*}
	\frac{\left(\overline{\kappa}x-\!\sqrt{\kappa\overline{\kappa}}\,y-\!\sqrt{\kappa}\,\right)^{\!a} \!\left(\kappa y-\!\sqrt{\kappa\overline{\kappa}}\,x-\!\sqrt{\overline{\kappa}}\,\right)^{\!b}}{\sqrt{a!b!}\,\omega^{a+b}}
\end{equation*}
times
\begin{multline*}
	\sum_{\ell_1,\ell_2,\ell_3,\ell_4,\ell_5=0}^{\infty} \!\!\!\!\!\!\! \frac{(-a)_{\ell_1+2\ell_3+\ell_5}(-b)_{\ell_2+2\ell_4+\ell_5}\bigl(-\frac{\sqrt{\kappa}\,x+\sqrt{\overline{\kappa}}\,y+1}{\omega}\bigr)_{\ell_1+\ell_2}}{\ell_1!\ell_2!\ell_3!\ell_4!\ell_5!} \\
	\times \frac{(-\frac{1}{2})^{\ell_3+\ell_4}\sqrt{\kappa}^{\,\ell_1+2\ell_4+\ell_5} \sqrt{\overline{\kappa}}^{\,\ell_2+2\ell_3+\ell_5} \omega^{\ell_1+\ell_2+\ell_3+\ell_4+\ell_5}}{\left(\overline{\kappa}x-\!\sqrt{\kappa\overline{\kappa}}\,y-\!\sqrt{\kappa}\,\right)^{\!\ell_1+2\ell_3+\ell_5}\!\left(\kappa y-\!\sqrt{\kappa\overline{\kappa}}\,x-\!\sqrt{\overline{\kappa}}\,\right)^{\!\ell_2+2\ell_4+\ell_5}}.
\end{multline*}
This expression is not of Aomoto--Gelfand type, but is still a terminating hypergeometric series in the sense of Horn.

We next obtain five-term recurrence relations for the $\mathcal{CH}_{a,b}$ as limits of those for the $\hat{\mathcal{K}}_{a,b}$.
Note that
\begin{equation*}
	A\overline{A}=A(J-A-I)=(\overline{k}-\overline{\mu})A+(k-\mu)\overline{A}
\end{equation*}
by \eqref{quadratic equation} and \eqref{parameters of complement}.
From this, \eqref{adjacency matrix of power}, \eqref{quadratic equation}, and \eqref{adjacency matrices of extension}, it follows that
\begin{align*}
	\bm{A}\bm{A}_{a,b} &= (a+1)\bm{A}_{a+1,b} +(a+1)(\overline{k}-\overline{\mu})\bm{A}_{a+1,b-1} + (a\lambda+b(k-\mu))\bm{A}_{a,b} \\
	& \quad +(b+1)\mu\bm{A}_{a-1,b+1} +(n-a-b+1)k\bm{A}_{a-1,b}, \\
	\overline{\bm{A}}\bm{A}_{a,b} &= (b+1)\bm{A}_{a,b+1} +(a+1)\overline{\mu}\bm{A}_{a+1,b-1} +(a(\overline{k}-\overline{\mu})+b\overline{\lambda}) \bm{A}_{a,b} \\
	& \quad +(b+1)(k-\mu)\bm{A}_{a-1,b+1} +(n-a-b+1)\overline{k}\bm{A}_{a,b-1}.
\end{align*}
Observe that the above identities correspond to the expansions of
\begin{equation*}
	\theta_{j,h}\cdot k_{a,b}\hspace{.01in} \mathcal{K}_{a,b}(j,h), \quad \overline{\theta}_{j,h}\cdot k_{a,b}\hspace{.01in} \mathcal{K}_{a,b}(j,h),
\end{equation*}
respectively, in terms of the polynomials $k_{c,d}\hspace{.01in} \mathcal{K}_{c,d}$ $((c,d)\in\Lambda_n)$.
See also \cite[Section 6]{IT2012TAMS}.
By \eqref{x,y}, we now rewrite these as recurrence relations for the $\hat{\mathcal{K}}_{a,b}=\sqrt{\vphantom{k}\smash{k_{a,b}}}\,\mathcal{K}_{a,b}$ and then let $n\rightarrow\infty$.
For example, the coefficient of $\hat{\mathcal{K}}_{a,b}$ in $x\hspace{.01in}\hat{\mathcal{K}}_{a,b}$ is given by
\begin{equation}\label{central coefficient}
	\frac{a\lambda+b(k-\mu)}{\sqrt{nk}}=\frac{a(r+s+k+rs)-brs}{k}\sqrt{\frac{k}{n}},
\end{equation}
where we have used \eqref{lamda,mu in terms of k,r,s}.
We claim that this coefficient converges to
\begin{equation*}
	(a\kappa+b\overline{\kappa})\frac{\sqrt{\kappa}}{\omega}.
\end{equation*}
To see this, we note that
\begin{equation*}
	-rs(f+g)=kv-k^2+(r+s)k
\end{equation*}
by virtue of \eqref{linear traces} and \eqref{quadratic trace 1}, from which it follows that (cf.~\eqref{sum to v})
\begin{equation}\label{curious limit}
	-\frac{rs}{k}\rightarrow\frac{\overline{\kappa}}{\omega}.
\end{equation}
If $\kappa>0$ then the claim follows directly from \eqref{curious limit}.
On the other hand, if $\kappa=0$ then the coefficient \eqref{central coefficient} converges to zero since the first factor of the RHS is bounded by \eqref{range of eigenvalues} and \eqref{curious limit}, and hence the claim again holds.
The other coefficients can be computed similarly and more easily using \eqref{lamda,mu in terms of k,r,s}, \eqref{degree}, and \eqref{curious limit} (and the complement versions of \eqref{lamda,mu in terms of k,r,s} and \eqref{curious limit}).
It follows that
\begin{align*}
	x\,\mathcal{CH}_{a,b} &= \sqrt{a+1}\,\mathcal{CH}_{a+1,b} +\!\sqrt{(a+1)b}\,\frac{\kappa\sqrt{\overline{\kappa}}}{\omega}\,\mathcal{CH}_{a+1,b-1} \\
	& \quad + (a\kappa+b\overline{\kappa})\frac{\sqrt{\kappa}}{\omega}\,\mathcal{CH}_{a,b} +\!\sqrt{a(b+1)}\,\frac{\kappa\sqrt{\overline{\kappa}}}{\omega}\,\mathcal{CH}_{a-1,b+1} +\!\sqrt{a}\,\mathcal{CH}_{a-1,b}, \\
	y\,\mathcal{CH}_{a,b} &= \sqrt{b+1}\,\mathcal{CH}_{a,b+1} +\!\sqrt{(a+1)b}\,\frac{\overline{\kappa}\sqrt{\kappa}}{\omega}\,\mathcal{CH}_{a+1,b-1} \\
	& \quad +(a\kappa+b\overline{\kappa})\frac{\sqrt{\overline{\kappa}}}{\omega}\,\mathcal{CH}_{a,b} +\!\sqrt{a(b+1)}\,\frac{\overline{\kappa}\sqrt{\kappa}}{\omega}\,\mathcal{CH}_{a-1,b+1} +\!\sqrt{b}\,\mathcal{CH}_{a,b-1}
\end{align*}
for $a,b=0,1,2,\dots$.
In particular, it follows inductively that the $\mathcal{CH}_{a,b}$ form a linear basis of $\mathbb{C}[x,y]$.
More precisely, the coefficients of every monomial $x^cy^d$ in terms of the $\mathcal{CH}_{a,b}$ are the limits of those of $x^cy^d$ in terms of the $\hat{\mathcal{K}}_{a,b}$.

We are now ready to establish the orthogonality relation for the $\mathcal{CH}_{a,b}$.
Let $\nu$ be the probability measure in Theorem \ref{main theorem}\,(iii), and recall that $\nu_{G^{\square n},\,\myoverline{G}^{\square n}}$ converges weakly to $\nu$.
Pick any monomial $x^cy^d$.
Observe that
\begin{equation*}
	\int_{|x^cy^d|\geqslant L} \bigl|x^cy^d\bigr|\,\nu_{G^{\square n},\,\myoverline{G}^{\square n}}(dxdy) \leqslant \frac{1}{L}\int_{\mathbb{R}^2} (x^cy^d)^2\,\nu_{G^{\square n},\,\myoverline{G}^{\square n}}(dxdy)
\end{equation*}
for every $L>0$, and that the integral in the RHS is uniformly bounded with respect to $n$ by virtue of \eqref{orthogonality for Krawtchouk} and the above comment.
Hence we have
\begin{equation*}
	\lim_{L\rightarrow\infty}\sup_n \int_{|x^cy^d|\geqslant L} \bigl|x^cy^d\bigr|\,\nu_{G^{\square n},\,\myoverline{G}^{\square n}}(dxdy)=0.
\end{equation*}
It is well known (see e.g., \cite[Lemma 8.4.3]{Bogachev2007B}) that this implies that
\begin{equation*}
	\int_{\mathbb{R}^2} x^cy^d\,\nu_{G^{\square n},\,\myoverline{G}^{\square n}}(dxdy)\rightarrow \int_{\mathbb{R}^2} x^cy^d\,\nu(dxdy).
\end{equation*}
Combining this with \eqref{orthogonality for Krawtchouk}, it follows that
\begin{equation}\label{orthogonality for Charlier-Hermite}
	\int_{\mathbb{R}^2} \! \mathcal{CH}_{a,b}(x,y)\hspace{.01in}\mathcal{CH}_{c,d}(x,y)\,\nu(dxdy)=\delta_{a,c}\delta_{b,d}
\end{equation}
for $a,b,c,d=0,1,2,\dots$.

\begin{rem}\label{OPs for other cases}
The discussions in this section work for the other cases in Theorem \ref{main theorem} as well, and we obtain (special cases of) the \emph{bivariate Charlier polynomials} studied by Genest, Miki, Vinet, and Zhedanov \cite{GMVZ2014JPA} for Theorem \ref{main theorem}\,(i) and (ii), and the products of univariate Hermite polynomials for Theorem \ref{main theorem}\,(iv); cf.~\cite[Section 5.1.3]{DX2014B}.
We omit the details.
Genest et al.~\cite[Section 10]{GMVZ2014JPA} also showed among other results that the bivariate Charlier polynomials can be obtained from the bivariate Krawtchouk polynomials by a limit process.
It should be remarked that, unless $\kappa=0$ or $\overline{\kappa}=0$, the bivariate Charlier--Hermite polynomials $\mathcal{CH}_{a,b}$ differ from affine transformations of the products of univariate Poisson and Hermite polynomials, which form another system of orthogonal polynomials with respect to the measure $\nu$ in Theorem \ref{main theorem}\,(iii); cf.~\eqref{linear CH polynomials}.
\end{rem}

\begin{rem}
The above proof of the orthogonality relation \eqref{orthogonality for Charlier-Hermite} for the $\mathcal{CH}_{a,b}$ is based on the weak convergence $\nu_{G^{\square n},\,\myoverline{G}^{\square n}}\rightarrow\nu$, and Example \ref{many TDs} now guarantees that \eqref{orthogonality for Charlier-Hermite} is valid for all the parameters $\kappa,\overline{\kappa}\geqslant 0$ with $\omega=\kappa+\overline{\kappa}>0$.
On the other hand, the bivariate Krawtchouk polynomials $\mathcal{K}_{a,b}$ and their orthogonality relation (cf.~\eqref{orthogonality for Krawtchouk}) have been discussed at a purely algebraic/parametric level \cite{IT2012TAMS,MT2004PAMS}, and it seems that we may also establish \eqref{orthogonality for Charlier-Hermite} in full generality by taking limits at this level and therefore without any reference to specific constructions of strongly regular graphs as in Example \ref{many TDs}.
\end{rem}

\section*{Acknowledgments}

This work, except Section \ref{sec: Charlier-Hermite polynomials}, is part of the Ph.D. thesis of JVSM \cite{Morales2017D}.
NO was supported by JSPS KAKENHI Grant Number JP16H03939.
HT was supported by JSPS KAKENHI Grant Numbers JP25400034 and JP17K05156.

\end{document}